\documentclass[a4paper,12pt,twoside]{amsart}

\usepackage[cm]{fullpage}

\usepackage{mymacros,oitp}

\usepackage[whole,autotilde]{bxcjkjatype}
\usepackage[unicode,bookmarks,pagebackref]{hyperref}

\begin{document}

\title[]{Moduli of K3 surfaces of degree 2
with four rational double points of type $D_4$}
\author[K.~Ueda]{Kazushi Ueda}
\address{
Graduate School of Mathematical Sciences,
The University of Tokyo,
3-8-1 Komaba,
Meguro-ku,
Tokyo,
153-8914,
Japan.}
\email{kazushi@ms.u-tokyo.ac.jp}
\date{}
\pagestyle{plain}

\begin{abstract}
We show that
the Satake--Baily--Borel compactification
of the moduli space
of lattice polarized K3 surfaces
parametrizing K3 surfaces of degree 2
with four rational double points of type $D_4$
is the projective 3-space.
We also show that
the corresponding graded ring of automorphic forms
is generated by four elements of weight 2
and one element of weight 11
with one relation of weight 22.
\end{abstract}

\maketitle

\section{Introduction}

It is a classical fact
that the maximum number of nodes
of a quartic surface in $\bP^3$
is sixteen.
Indeed,
if $Y$ is a quartic surface with a node $p$,
then the projection from $p$
gives a double cover of $\bP^2$.
The branch locus is a sextic in $\bP^2$,
and the set of nodes of $Y$ other than $p$ is bijective
with the set of nodes of the sextic.
The maximum number of nodes of a sextic is fifteen,
which is attained if the sextic is the union of six lines
no three of which intersect at one point.
The double cover of $\bP^2$ branched along six lines
comes from a quartic surface with sixteen nodes
if and only if no three lines intersect at one point
and all lines are tangent to a smooth conic.

A quartic surface with sixteen nodes is called a \emph{Kummer quartic surface}.
The moduli space of Kummer quartic surfaces,
equipped with some additional combinatorial data
(which is equivalent to a level 2 structure
on the corresponding principally polarized Abelian surface,
which in turn is equivalent to the choice of a total order
on the set of six Weierstrass points of the corresponding genus 2 curve),
is an open subset of the Igusa quartic.
The Igusa quartic,
also known as the Castelnuovo--Richmond quartic
(see \cite{MR2964027}),
is the Satake--Baily--Borel compactification
of the Siegel modular variety
$\cA_2(2)$
\cite{MR168805}.
The Igusa quartic contains fifteen singular lines,
and for a smooth point on the Igusa quartic,
the corresponding Kummer quartic surface is the intersection
of the Igusa quartic and its tangent space at that point,
which has one node at that point
and fifteen other nodes at the intersection points with singular lines.

In this paper,
we classify Gorenstein K3 surfaces
of degree 2
with the maximum number of rational double points of type $D_4$.
Since the signature of the K3 lattice $\LambdaK$
is $(3,19)$,
this number cannot be more than four,
and it is easy to find a Gorenstein K3 surface of degree 2
with four rational double points of type $D_4$.
Let $P$ be
the Picard lattice of the minimal resolution
of a very general Gorenstein K3 surface of degree 2
with four rational double points of type $D_4$,
and $Q$ be the 
orthogonal complement of $P$
in the K3 lattice $\LambdaK$.
Instead of the subspace of the moduli space of K3 surfaces of degree 2,
we consider the moduli space of
K3 surfaces
polarized by the lattice $P$
in the sense of Nikulin
(see \pref{df:Nikulin}),
which is an orthogonal modular variety
for the lattice $Q$.
The main result is \pref{th:main},
which describes this modular variety
and the corresponding graded ring of automorphic forms.
The proof of \pref{th:main} occupies
Sections
\ref{sc:elliptic},
\ref{sc:H10},
\ref{sc:P3},
and
\ref{sc:proof}.
In \pref{sc:2+E8+E8},
we classify K3 surfaces of degree 2
with two rational double points of type $E_8$.
This gives a proof
of a theorem of Igusa
\cite{MR168805}
describing
the graded ring of Siegel modular forms
of genus 2
in terms of generators and relations.
Note that the relation between
lattice polarized K3 surfaces
and Siegel modular forms
is also discussed
in \cite{MR2935386}
using the same lattice
and a different model
(i.e.,
a family of quartic surfaces
which is related to the family
appearing in this paper
by a birational map).


\section{Gorenstein K3 surfaces of degree 2}

A \emph{Gorenstein K3 surface of degree 2} is a pair $(Y,\cL)$
consisting of a Gorenstein projective surface $Y$
and an ample line bundle $\cL$ of degree 2 on $Y$
such that the minimal model $\mu \colon \Yt \to Y$ is a K3 surface.
Set $\cLt \coloneqq \mu^* \cL$.
If $(Y, \cL)$ is a Gorenstein K3 surface of degree 2,
then it follows from \cite{MR330172}
(see also \cite{MR595204})
that either
\begin{itemize}
\item
$|\cL|$ is base point free and
$\varphi_L \colon Y \to |\cL|$
is a double cover of $\bP^2$,
or
\item
the fixed component $s$ of $|\cLt|$
is a $(-2)$-curve,
and the free part is $2f$
where $f$ is a curve of self-intersection number $0$.
\end{itemize}
In the former case,
the branch locus must be a sextic.
In the latter case,
the morphism
$\varphi_f \colon \Yt \to |f| \cong \bP^1$
is an elliptic fibration
with a section $s$.
In both cases,
the surface can be modeled
by a complete intersection
of bidegree $(2,6)$
in $\bP(1,1,1,2,3)$
of the form
\begin{align} \label{eq:(2,6)-complete intersection}
 a_0 x_4 = g_2(x_1,x_2,x_3), \qquad
 x_5^2 = f_6(x_1,x_2,x_3,x_4),
\end{align}
where
$a_0$, $g_2$, and $f_6$ are homogeneous polynomials
of degrees 0, 2, and 6
(the linear term in $x_5$
is eliminated by a coordinate transformation).
If $a_0 \ne 0$,
then $Y$ is
the double cover of $\bP^2_{[x_1:x_2:x_3]}$
branched along the sextic curve
defined by $f_6(x_1,x_2,x_3,g_2(x_1,x_2,x_3)/a_0)$.
If $a_0 = 0$,
then the projection to $\bP^2_{[x_1:x_2:x_3]}$
gives a rational map
to the conic defined by $g_2 = 0$,
which becomes regular by blowing up the point $\ld 0:0:0:1:\sqrt{f_6(0,0,0,1)} \rd$.
The exceptional divisor of the blow-up is the section $s$.

\section{Double covers of $\bP^2$}

Let
\begin{align} \label{eq:B}
B \coloneqq
\lc
p_1 \coloneqq [0:0:1], \ 
p_2 \coloneqq [0:1:0], \ 
p_3 \coloneqq [1:0:0], \ 
p_4 \coloneqq [1:1:1] \rc
\end{align}
be a set of four point on $\bP^2$ in general position,
which is unique up to the action of $\Aut \bP^2 \cong \PGL_3(\bC)$.

\begin{lemma}
The space $S_B$ of homogeneous polynomials $f$ of degree 6 in $\bC[x_1,x_2,x_3]$,
such that the double cover of $\bP^2$
branched along the zero of $f$
has singularities
which are equal to or worse than rational double points of type $D_4$
above $B$,
is a vector space of dimension 4.
\end{lemma}

\begin{proof}
Let $Y$ be the double covers of $\bP^2_{[x_1:x_2:x_3]}$
branched along the sextic curve $C$ defined by
\begin{align}
f(x_1,x_2,x_3) = \sum_{i+j+k=6} a_{ijk} x_1^i x_2^j x_3^k.
\end{align}
Then $Y$ has a singularity
which is equal to or worse than a rational double point of type $D_4$
above a point $p \in \bP^2_{[x_1:x_2:x_3]}$
if and only if
$C$ has a singularity
which is equal to or worse than a curve singularity of type $D_4$ at $p$,
which is the case
if and only if the Hesse matrix
\begin{align}
\Hess(f) \coloneqq \lb \frac{\partial^2 f}{\partial x_i \partial x_j} \rb_{i,j=1}^3
\end{align}
vanishes at $p$.
Note that for each point $p \in \bP^2$,
the condition $\Hess(f)(p) = 0$
consists of six linear equations on $(a_{ijk})_{i+j+k=6} \in \bC^{28}$.
One can easily see that $\Hess(f)(p_1) = 0$
if and only if
$
a_{i,j,k} = 0 
$
for all $(i,j,k)$
satisfying $i+j \le 2$.
Repeating the same argument,
one shows that
$\Hess(f) (p_1) = \Hess(f) (p_2) = \Hess(f) (p_3) = 0$
if and only if
\begin{align} \label{eq:f}
  f &= a_{330} x^3 y^3 + a_{303} x^3 z^3 + a_{033} y^3 z^3
+ a_{222} x^2 y^2 z^2 \\
&\qquad + a_{321} x^3 y^2 z + a_{231} x^2 y^3 z
+ a_{312} x^3 y z^2 + a_{213} x^2 y z^3
+ a_{132} x y^3 z^2 + a_{123} x y^2 z^3. \nonumber
\end{align}
If we further impose $\Hess(f)(p_4)=0$,
then the coefficients $a_{ijk}$ appearing
on the second line of \pref{eq:f}
are determined
by those appearing on the first line.
\end{proof}

\begin{lemma} \label{lm:general position}
If the double cover of $\bP^2$
branched along the zero of a homogeneous polynomial $f$ of degree 6
has singularities
which are equal to or worse than rational double points of type $D_4$
above four points
three of which are colinear,
then it has a singularity
worse than rational double points.
\end{lemma}

\begin{proof}
A direct calculation shows that
if one imposes
$\Hess(f)(1,1,0) = 0$
on \pref{eq:f},
then one obtains $a_{321}=a_{231}=0$
(and $a_{330} = a_{222}+a_{312}+a_{132} = 0$),
so that $f$ is divisible by $z^2$.
It follows that the singularity is not isolated,
and hence not a rational double point in particular.
\end{proof}

Let
\begin{align}
Q_B
\coloneqq \lc a_1 x_2 x_3 + a_2 x_1 x_3 + a_3 x_1 x_2
\in \bC[x_1,x_2,x_3] \relmid a_1 + a_2 + a_3 = 0 \rc
\end{align}
be the space of homogeneous polynomials of degree 2
vanishing at $B$.
The projective line $\bP(Q_B)$ can be identified
with the pencil $|2H-B|$ of conics
passing through $B$.
The pencil contains three singular members
\begin{align} \label{eq:D}
D_1 \coloneqq 
\overline{p_1 p_2} \cup \overline{p_3 p_4}, \quad
D_2 \coloneqq 
\overline{p_1 p_3} \cup \overline{p_2 p_4}, \quad
D_3 \coloneqq 
\overline{p_1 p_4} \cup \overline{p_2 p_3},
\end{align}
each of which
is the union of two lines.

\begin{lemma} \label{lm:QB to SB}
The map
\begin{align} \label{eq:QB to SB}
(Q_B)^3 \to \bC[x_1,x_2,x_3],
\qquad
(q_1,q_2,q_3) \mapsto q_1 q_2 q_3
\end{align}
induces a surjection
$
\bP(Q_B)^3 \to \bP(S_B),
$
which can be identified with
the natural projection
$
(\bP^1)^3 \to (\bP^1)^3/\fS_3
\cong \bP^3
$
to the symmetric product.
\end{lemma}

\begin{proof}
If we choose $(a_{033},a_{303},a_{330},a_{222})$ as a coordinate of $S_B$,
$(a_1,a_2,a_3)$ with $a_1+a_2+a_3=0$ as a coordinate of $Q_B$
and write $q_i \in Q_B$ as $(a_{i1},a_{i2},a_{i3})$ for $i=1,2,3$,
then
the map \pref{eq:QB to SB} sends $(q_i)_{i=1}^3$ to
$
\lb
a_{11} a_{21} a_{31},
a_{12} a_{22} a_{32},
a_{13} a_{23} a_{33},
\sum_{\sigma \in \fS_3} a_{1\sigma(i)} a_{2 \sigma(2)} a_{3 \sigma(3)}
\rb.
$
\end{proof}

\begin{remark}
For $i \in \{1,2,3\}$,
let $Q_i$
be the conic defined by $q_i$.
The cremona transformation
$[x_1:x_2:x_3] \mapsto [1/x_1:1/x_2:1/x_3]$
of $\bP^2$ sends the union $C = Q_1 \cup Q_2 \cup Q_3$
of three conics
to the union of six lines defined by
\begin{align}
x_1 x_2 x_3 \prod_{i=1}^3 \lb a_{i1} x_1 + a_{i2} x_2 + a_{i3} x_3 \rb,
\end{align}
which are not in general position
but three lines intersects at $[1:1:1]$.
The double cover of $\bP^2$ branched along it
has twelve nodes and one $D_4$-singularity.
\end{remark}

\section{Lattices and orthogonal groups} \label{sc:lattices}

We use the following notations in this paper:
\begin{itemize}
\item
For a pair $(m,n)$ of natural numbers,
the lattice $\bfI_{m,n}$ is the free abelian group of rank $m+n$
equipped with the inner product
\begin{align}
((a_i)_i,(b_i)_i) = \sum_{i=1}^m a_i b_i - \sum_{i=m+1}^{n+m} a_i b_i.
\end{align}
\item
For a pair $(m,n)$ of positive integers satisfying $m \equiv n \mod 8$,
the lattice $\bfII_{m,n}$ is the unique even unimodular indefinite lattice
of signature $(m,n)$.
\item
The root lattice of type $X_k$ is denoted
by $\bfX_k$.
\item
For an integer $k$,
the lattice $\la k \ra$ is generated
by a single element with self-intersection $k$.
\item
The lattice $\bfU$ is the even unimodular hyperbolic plane $\bfII_{1,1}$.
\item
Given a lattice $L$,
the lattice $L(k)$ has the same underlying free abelian group as $L$
and the inner product which is $k$ times that of $L$.
\item
The K3 lattice
$
\LambdaK
\cong \bfII_{3,19}
\cong \bfE_8^{\bot 2} \bot \bfU^{\bot 3}
$
is the second integral cohomology group of a K3 surface
equipped with the intersection form.
\item
Given a lattice $L$,
the set of $(-2)$-elements is denoted by
$
\Delta(L) \coloneqq \lc \delta \in L \relmid (\delta, \delta) = -2 \rc.
$
\end{itemize}

The \emph{discriminant group}
of a non-degenerate lattice $L$
is the quotient
$
D_L
$
of the dual free abelian group
$
L^\dual \coloneqq \Hom(L, \bZ)
$
by the image of the natural map
$
L \hookrightarrow L^\dual,
$
$
l \mapsto (l, -).
$
The $\ell$-invariant $\ell_L$
is defined as the minimal number of generators
of the finite abelian group $D_L$.
A non-degenerate lattice $L$ is said to be \emph{2-elementary}
if
$
D_L \cong (\bZ/2\bZ)^{\ell_L}
$.
If the lattice $L$ is even,
then the \emph{discriminant form} $q_L$
is defined as
the quadratic form
on the discriminant group
sending
$
x \in D_L
$
to
$
[(x,x)] \in \bQ/2\bZ.
$
The \emph{parity invariant} of $L$ is defined by
\begin{align}
\delta_L
\coloneqq
\begin{cases}
0 & \text{$q_L$ takes values in $\bZ/2\bZ \subset \bQ/2\bZ$}, \\
1 & \text{otherwise}.
\end{cases}
\end{align}
By Nikulin \cite{MR525944}
(cf.~also \cite[Theorem 1.5.2]{MR728992}),
the isomorphism class of an even indefinite 2-elementary lattice
is determined uniquely
by
the parity invariant,
the $\ell$-invariant,
and
the signature.

The \emph{stable orthogonal group}
$
\rOt(L)
$
is
defined as
the kernel of the natural projection
from
$
\rO(L)
$
to the finite orthogonal group
$
\rO(D_L, q_L).
$

Let $P$ be
the Picard lattice of the minimal resolution
of a very general Gorenstein K3 surface of degree 2
with four rational double points of type $D_4$,
and $Q$ be the 
orthogonal complement of $P$
in the K3 lattice $\LambdaK$.

\begin{lemma} \label{lm:P}
The lattice $P$ is an extension of
$
\lb \bfD_4 \rb^{\bot 4} \bot \la 2 \ra
$
by $(\bZ/2\bZ)^2$.
\end{lemma}

\begin{proof}
A very general Gorenstein K3 surface of degree 2
with four rational double point of type $D_4$
is a double cover $\varphi \colon Y \to \bP^2$
branched along the union $Q_1 \cup Q_2 \cup Q_3$
of three smooth conics
passing through four points
$
B = \{ p_i \}_{i=1}^4 \subset \bP^2
$
in general position.
Let
$\{ E_{ij} \}_{i,j=1}^4$
be the set of irreducible components
of the exceptional divisor
of the minimal resolution
$
\mu \colon \Yt \to Y
$
such that $\mu(E_{ij}) = p_i$
and
\begin{align}
E_{ij} \cdot E_{ik} =
\begin{cases}
-2 & j = k, \\
1 & (j,k) \in \{ (1,4), (2,4), (3,4) \}, \\
0 & \text{otherwise}
\end{cases}
\end{align}
for any $i \in \{ 1, \ldots, 4 \}$.
The total transform $\ell$
of the hyperplane $H$ in $\bP^2$
and the set $\{ E_{ij} \}_{i,j=1}^4$ of divisors 
generates a sublattice $P_0$ of $P \coloneqq \Pic \Yt$
isometric to $\la 2 \ra \bot \bfD_4^{\bot 4}$.

The strict transform of the conic $Q_i$,
which is linearly equivalent to
\begin{align}
2 \ell - \sum_{j=1}^4 \lb E_{ji} + E_{j1} + E_{j2} + E_{j3} + 2 E_{j4} \rb,
\end{align}
is twice a $(-2)$-curve $\Qbar_i$.
The group $P/P_0$
is isomorphic to $\lb \bZ / 2 \bZ \rb^2$,
and $\lc \ld \Qbar_i \rd \rc_{i=1,2,3} \subset P/P_0$
is the set of non-zero elements.
\end{proof}

\begin{lemma} \label{lm:Q}
The lattice $Q$ is isometric to
$
\bfI_{2,3}(2)
$
\end{lemma}

\begin{proof}
Since
$
\lb \bfD_4 \rb^{\bot 4} \bot \la 2 \ra
$
is 2-elementary
with the $\ell$-invariant 9
and the parity invariant 1,
the lattice $P$ is 2-elementary
with $\ell_P=5$
and $\delta_P=1$.
Hence the lattice $Q$ is characterized
as the unique even 2-elementary lattice
with $\delta_Q = 1$,
$\ell_Q = 5$,
and signature $(2,3)$.
\end{proof}

\begin{remark}
It follows
from \pref{lm:Q}
that
$Q^\dual$ is identified with $\frac{1}{2} Q \subset Q \otimes_\bZ \bQ$,
so that
$
D_Q
\coloneqq Q^\dual/Q
\cong \frac{1}{2} Q / Q
\cong Q / 2 Q
$
and
$
\rOt(Q) = \lc g \in \rO(Q) \relmid g \equiv \id \mod 2 \rc.
$
\end{remark}

\section{Automorphic forms and modular varieties for orthogonal groups}

Let
$
\cDt
$
be a connected component of
$
\lc
\Omega \in Q \otimes \bC
\relmid
(\Omega, \Omega) = 0, \ 
(\Omega, \Omegabar) > 0
\rc
$
and $\cD$ be its image in $\bP(Q \otimes \bC)$,
which is a symmetric domain of type IV.
The subgroup of the orthogonal group $\rO(Q)$ of index two
preserving the connected component $\cDt$
will be denoted by $\rO^+(Q)$.
Given a subgroup $\Gamma$ of $\rO^+(Q)$,
an \emph{automorphic form} of weight $k \in \bZ$
is a holomorphic function
$
 f \colon \cDt \to \bC
$
satisfying
\begin{enumerate}[(i)]
 \item
$f(\alpha z) = \alpha^{-k} f(z)$
for any $\alpha \in \bCx$, and
 \item
$f(\gamma z) = f(z)$
for any
$
\gamma \in \Gamma
$.
\end{enumerate}
The space
$A_k(\Gamma)$
of automorphic forms of weight $k$
constitute the graded ring
$
A(\Gamma)
\coloneqq
\bigoplus_{k=0}^\infty
A_k(\Gamma),
$
whose projective spectrum
$
\cMbar(\Gamma)
\coloneqq
\Proj A(\Gamma)
$
is the Satake--Baily--Borel compactification
of the \emph{modular variety}
$
\cM(\Gamma) \coloneqq \cD / \Gamma.
$

We can choose a subset $\Delta(P)^+$ of
$
 \Delta(P)
$
satisfying
\begin{enumerate}
 \item
$\Delta(P) = \Delta(P)^+ \coprod (- \Delta(P)^+)$ and
 \item
if $\delta_1,\delta_2 \in \Delta(P)^+$ and $\delta_1+\delta_2 \in \Delta(P)$,
 then $\delta_1+\delta_2 \in \Delta(P)^+$.
\end{enumerate}
The choice of $\Delta(P)^+$ is unique
up to the action of $O(P)$.
Define
\begin{align}
 C(P) &\coloneqq \lc h \in P \relmid
  (h, \delta) \ge 0 \text{ for any } \delta \in \Delta(P)^+ \rc, \\
 C(P)^\circ &\coloneqq \lc h \in P \relmid
  (h, \delta) > 0 \text{ for any } \delta \in \Delta(P)^+ \rc.
\end{align}

For a projective variety $Y$, we set
\begin{align}
 \Pic(Y)^{+} &\coloneqq C(Y) \cap H^2(Y; \bZ), \\
 \Pic(Y)^{++} &\coloneqq C(Y)^\circ \cap H^2(Y; \bZ),
\end{align}
where $C(Y)^\circ \subset H^{1,1}(Y) \cap H^2(Y; \bR)$
is the K\"{a}hler cone of $Y$
and $C(Y)$ is its closure.

\begin{definition}[{Nikulin \cite{MR544937}}] \label{df:Nikulin}
A {\em $P$-polarized K3 surface} is a pair $(Y, j)$
where $Y$ is a K3 surface and
$
 j \colon P \hookrightarrow \Pic(Y)
$
is a primitive lattice embedding.
An {\em isomorphism} of $P$-polarized K3 surfaces
$(Y, j)$ and $(Y', j')$ is an isomorphism
$f : Y \to Y'$ of K3 surfaces
such that $j = f^* \circ j'$.
A $P$-polarized K3 surface is
\emph{pseudo-ample} if
$
 j(C(P)^\circ) \cap \Pic(Y)^+ \ne \emptyset,
$
and
\emph{ample} if
$
 j(C(P)^\circ) \cap \Pic(Y)^{++} \ne \emptyset.
$
\end{definition}

As explained in \cite[Remark 3.4]{MR1420220},
the coarse moduli space $\cM$ of
pseudo-ample $P$-polarized K3 surfaces
is isomorphic to $\cM \lb \rOt^+(Q) \rb$
where $\rOt^+(Q) \coloneqq \rO^+(Q) \cap \rOt(Q)$.
The coarse moduli space of ample $P$-polarized K3 surfaces
is the complement
$
\cM \setminus \cH
$
of the union
$
\cH =
\bigcup_{\delta \in \Delta^+(Q)} \cH_\delta
$
of the divisors
$
\cH_{\delta}
\coloneqq \left. \rOt^+(Q) \, {\delta}^\bot \middle/ \, \rOt^+(Q) \right.
$
where
$
{\delta}^\bot
\coloneqq \{ [\Omega] \in \cD \mid (\Omega, \delta) = 0 \}
$
is the reflection hyperplane.

The main result of this paper is the following:

\begin{theorem} \label{th:main}
The variety $\cMbar \lb \rOt^+(Q) \rb$ is isomorphic to $\bP^3$.
The graded ring $A \lb \rOt^+(Q) \rb$ is generated by four elements of weight 2
and one element of weight 11
with one relation of weight 22.
\end{theorem}

\section{Elliptic K3 surfaces} \label{sc:elliptic}

Let $\cH_{10}$ be the subvariety
of the moduli space $\cM$
of pseudo-ample $P$-polarized K3 surfaces $(Y, j)$
parametrizing those
containing a $(-2)$-curve $s$
and an elliptic curve $f$
such that $\cL = \cO_Y(s + 2 f)$.
The notation $\cH_{10}$ is in anticipation
of the divisors $\cH_i$ for $i=1,\ldots,9$
appearing in \pref{sc:H10} below.
For such $(Y,j)$,
the sublattice $j \lb \lb \bfD_4 \rb^{\bot 4} \rb$
of $\Pic Y$
must be generated by irreducible components
of fibers,
so that $Y$ must have at least four singular fibers
of Kodaira type $I_0^*$.
This implies that $Y$ is a Kummer surface of product type,
i.e., the minimal model of the quotient
$
(E \times E') / (-1) \times (-1)
$
of the product of a pair of elliptic curves,
and the elliptic fibration comes from the projection
to $E / (-1) \cong \bP^1$.
Such a surface is modeled in \pref{eq:(2,6)-complete intersection}
by $a_0 = 0$,
$
g_2 \in |2H-B| \setminus \{ D_i \}_{i=1}^3
$,
and
\begin{align}
  f_6 = (x_4 + \lambda_1 g_2')(x_4 + \lambda_2 g_2')(x_4 + \lambda_3 g_2')
\end{align}
where
$(\lambda_1, \lambda_2, \lambda_3)$
is a configuration
of three ordered points on $\bA^1$
(the order determines and is determined by the $P$-polarization,
and one can set $(\lambda_1, \lambda_2, \lambda_3) = (0,1,\lambda)$
by translation and rescaling of $x_4$),
and $g_2'$ is any element of $Q_B$
linearly independent from $g_2$.
The choice of $g_2'$ is irrelevant
since $Q_B / \bC \cdot g_2$ is one-dimensional,
and one can choose
the defining equation $d_1$
of $D_1$
as $g_2'$.
The transcendental lattice $T_Y$ of $Y$
is isometric to $T_{E \times E'}(2)$,
which is contained in $\bfU(2)^{\bot 2}$.
The class of the $(-2)$-curve $s$
in $Q \subset H^2(Y;\bZ)$
is a $(-2)$-element,
which we write as $v_{10}$.
The orthogonal lattice of $v_{10}$
in $Q$ is $\bfU(2)^{\bot 2}$,
and $\cH_{10}$ is the divisor
associated with $v_{10}$.
The divisor $\cH_{10}$ is an orthogonal modular variety
for the lattice $\bfU(2)^{\bot 2}$,
which is the product of two copies of the modular curve
$
X(2) \cong \bP^1 \setminus \{ 0, 1, \infty \}
$
of level 2.

\section{The complement of $\cH_{10}$} \label{sc:H10}

Recall from the proof of \pref{lm:P}
that $\ell$ is the generator of $\la 2 \ra$
in $(\bfD_4)^{\bot 4} \bot \la 2 \ra \subset P$.
Let $(\Yt, j)$ be a $P$-polarized K3 surface
such that $\Yt \to |j(\ell)| \cong \bP^2$
is a morphism of degree two.
Then $\Yt$ has four $D_4$-configurations of $(-2)$-curves
which goes to $D_4$-singularities of the branch curve in $\bP^2$.
The four singular points must be in general position by \pref{lm:general position},
so that they can be set to $B$ defined in \pref{eq:B}
by the action of $\Aut \bP^2$.
Note also that the order of the four singular points is determined by the $P$-polarization.
The branch locus must be the union of three distinct conics
$
Q_1, Q_2, Q_3 \in |2H-B| \cong \bP^1
$,
and the order of these three conics is also determined by the $P$-polarization.
Conversely,
any ordered triple $(Q_1, Q_2, Q_3) \in |2H-B|^3 \setminus \Deltab$
determines a $P$-polarized K3 surface,
where $\Deltab \subset |2H-B|^3$ is the big diagonal
where at least two of the three conics are equal,
so that
\begin{align}
\cM \setminus \cH_{10} \cong |2H-B|^3 \setminus \Deltab.
\end{align}
Define divisors $\cH_i$ of $\cM$ for $i=1,\ldots,9$
as the closures in $\cM$ of
\begin{align}
(\cM \setminus \cH_{10}) \cap \cH_{3(i-1)+j}
\cong
\lc (Q_1, Q_2, Q_3) \in |2H-B|^3 \setminus \Deltab
\relmid Q_j = D_k
\rc
\end{align}
for $j,k \in \{1, 2, 3\}$,
where $D_k$ is the union of two lines
defined in \pref{eq:D}.
The Picard lattice of a very general K3 surface
on the divisor $\cH_i$
is an extension of $P \bot \bfA_1$ by $\bZ/2\bZ$,
where $\bfA_1$ comes from the resolution
of the $A_1$-singularity above the node of $D_k$,
and the extension comes from half the strict transform
of an irreducible component of $D_k$
just as in \pref{lm:P}.
The class of the exceptional divisor
for the resolution of the $A_1$-singularity
gives a $(-2)$-element $v_i$ in $Q \subset H^2(Y,\bZ)$,
and $\cH_i$ is the divisor of $\cM$
associated with $v_i$.
The transcendental lattice $v_i^\bot \subset Q$
is the even 2-primary lattice $\bfI_{2,2}(2)$
characterized by $\delta=1$,
$\ell=4$,
and the signature (2,2).

\section{The isomorphism $\cMbar \cong \bP^3$} \label{sc:P3}

Models of $P$-polarized K3 surfaces
in a neighborhood of $\cH_{10}$
can be parametrized by
$
(a_0, g_2, \lambda)
\in
\bC \times \lb |2H-B| \setminus \{ D_i \}_{i=1}^3 \rb \times \lb \bP^1 \setminus \{ 0, 1, \infty \} \rb
$
as
\begin{align}
a_0 x_4 = g_2, \quad
x_5^2 = x_4(x_4 + d_1)(x_4 + \lambda d_1).
\end{align}
When $a_0 \ne 0$,
one can eliminate $x_4$
so that
the corresponding $P$-polarized K3 surface
corresponds to the point
\begin{align}
(g_2, g_2 + a_0 d_1, g_2 + a_0 \lambda d_1)
\in
|2H-B|^3 \setminus \Deltab
\cong
\cM \setminus \cH_{10},
\end{align}
which goes to the small diagonal
\begin{align}
\Deltas = \lc (g_2, g_2, g_2) \in |2H-B|^3 \relmid g_2 \in |2H-B| \rc
\end{align}
as $a_0$ goes to 0.
It follows that the moduli space $\cM$
can be embedded in the blow-up $(|2H-B|^3)^\sim$
of $|2H-B|^3$ along $\Deltas$.
The complement
$
(|2H-B|^3)^\sim \setminus \cM
\cong
(\Deltab)^\sim \setminus (\Deltas)^\sim,
$
where $\Deltab^\sim$ is the pull-back of the big diagonal
and $\Deltas^\sim \cong \cH_{10}$ is the exceptional divisor,
consists of three irreducible components.
By contracting each of these three components to $\bP^1$,
one obtains $\bP^3$.
Hence $\cM$ is an open subvariety of $\bP^3$,
and the complement $\bP^3 \setminus \cM$ is the union of three lines.
Since the complement is of codimension 2,
$\bP^3$ must be isomorphic to the Satake--Baily--Borel compactification $\cMbar$.

To be explicit,
choose a homogeneous coordinate $[s:t]$ of $|2H-B|$
in such a way that $d_1$, $d_2$, and $d_3$ corresponds to
$[1:0]$, $[1:1]$, and $[0:1]$.
Then one can choose a homogeneous coordinate $[u_1:u_2:u_3:u_4]$
of $\cMbar \cong \bP^3$
in such a way that
the birational map
$
|2H-B|^3 \dashrightarrow \cMbar
$
sends
$
([s_1:t_1],[s_2:t_2],[s_3:t_3])
$
to
\begin{align}
[
s_1 s_2 t_3 - t_1 s_2 s_3
:
s_1 t_2 t_3 - t_1 t_2 s_3
:
s_1 s_2 t_3 - s_1 t_2 s_3
:
t_1 s_2 t_3 - t_1 t_2 s_3
].
\end{align}
The big diagonal is contracted as
\begin{align}
([s_1:t_1],[s_2:t_2],[s_1:t_1])
&\mapsto
([0:0:s_1:t_1]), \\
([s_1:t_1],[s_2:t_2],[s_2:t_2])
&\mapsto
([s_2:t_2:0:0]), \\
([s_1:t_1],[s_1:t_1],[s_3:t_3])
&\mapsto
([s_1:t_1:s_1:t_1]),
\end{align}
and the small diagonal is blown up to
\begin{align} \label{eq:quadric}
\lc [u_1:u_2:u_3:u_4] \in \bP^3 \relmid
u_1 u_4 = u_2 u_3 \rc
\cong \bP^1 \times \bP^1.
\end{align}
Under the embedding $\cM \subset \bP^3$,
the divisor $\cH_i$ is a hyperplane if $i \in \{ 1, \ldots, 9 \}$,
and the quadric \pref{eq:quadric} if $i=10$.


\section{Proof of \pref{th:main}}
\label{sc:proof}


Let
$
\bM
\coloneqq \ld \cD \middle/ \SOt^+(Q) \rd
\cong \ld \cDt \middle/ \SOt^+(Q) \times \bCx \rd
$
be the orbifold quotient,
where we use 
$\SOt^+(Q) \coloneqq \rOt^+(Q) \cap \SO(Q)$
instead of
$\rOt^+(Q)$,
since the subgroup
$\{ \pm \id_Q \} \subset \rOt^+(Q)$ acts
trivially on $\cD$,
so that one has
$
A_k \lb \rOt^+(Q) \rb
=
A_k \lb \SOt^+(Q) \rb
$
while
$
\ld \cD \middle/ \rOt^+(Q) \rd
$
has $\{ \pm \id_Q \}$ as the generic stabilizer.
For any $k \in \bZ$,
the line bundle
$
\cO_\cD(k) \coloneqq \cO_{\bP(Q)}(k)|_\cD
$
is invariant
under the action of $\SOt^+(Q)$,
and hence descends to a line bundle on $\bM$,
which we write as $\cO_\bM(k)$.
Then the line bundle $\cO_\bM(k)$ is associated
with the trivial line bundle on $\cDt$
equipped with the character
$
\SOt^+(Q) \times \bCx \ni (g, \alpha) \mapsto \alpha^k,
$
so that one has
\begin{align}
A_k \lb \SOt^+(Q) \rb
\cong
H^0 \lb \cO_\bM(-k) \rb.
\end{align}

Recall that the canonical bundle
of a hypersurface $X$
of degree $d$ in $\bP^n$
is given by $\omega_X \cong \cO_X(d-n-1)$.
Since $\cD$ is an open subvariety
defined by $(\Omega,\Omegabar) > 0$
of the quadric hypersurface
defined by $(\Omega, \Omega) = 0$
in the 4-dimensional projective space $\bP(Q \otimes \bC)$,
one has
\begin{align} \label{eq:omega_D}
\omega_\cD \cong \cO_\cD(-3).
\end{align}
Since the isomorphism \pref{eq:omega_D}
is equivariant with respect to the action of $\SOt(Q)$,
one has
\begin{align} \label{eq:omega_bM}
\omega_\bM \cong \cO_\bM(-3).
\end{align}

The structure morphism
$
\phi \colon \bM \to \cM
$
to the coarse moduli space
is an isomorphism in codimension one
outside the image $\bH \coloneqq \bigcup_{\delta \in \Delta(Q)} \bH_\delta$
of the reflection hyperplanes in $\cD$.
Locally near a smooth point of $\bH$,
the orbifold $\bM$ is the quotient of the double cover of $\cM$
branched along $\cH$
by the group of deck transformations
(known as the \emph{root construction}
\cite{MR2306040,MR2450211}
of $\cM$ along $\cH$),
so that
\begin{align} \label{eq:omega_bM2}
\omega_\bM
&\cong \lb \phi^* \omega_\cM \rb \lb \bH \rb
\end{align}
where
\begin{align} \label{eq:omega_cM}
\omega_\cM
\cong
\omega_{\bP^3}|_\cM
\cong
\cO_\cM(-4)
\end{align}
and
$
\cO_\bM(\bH)
$
is the `square root' of $\cO_\cM(\cH)$
(which is the raison d'etre of the root construction);
\begin{align} \label{eq:branch}
\cO_\bM(2 \bH)
\cong
\phi^* \cO_\cM(\cH).
\end{align}

\begin{lemma} \label{lm:weight}
One has $\phi^* \cO_\cM(-1) \cong \cO_\bM(2)$.
\end{lemma}

\begin{proof}
As explained in \pref{sc:P3},
the Satake--Baily--Borel compactification $\cMbar \cong \bP^3$
contains the closure
$\cHbar_{10} \cong \Xbar(2) \times \Xbar(2)$
of $\cH_{10}$
as a quadric hypersurface.
It follows that $\cO_{\cMbar}(1)$ restricts to
$\cO_{\Xbar(2)}(1) \boxtimes \cO_{\Xbar(2)}(1)$.
Fix $g_2 \in |2H-B| \setminus \{ D_i \}_{i=1}^3 \cong X(2)$
and consider the family
$
\varphi \colon \cY \subset \bP(\bfq) \times U \to U
$
of complete intersections
in the weighted projective space $\bP(\bfq)$
of weight
$\bfq=(q_1, \ldots, q_5) = (1,1,1,2,3)$
over
$U \coloneqq \Spec \bfk[\lambda_0^{\pm 1}, \lambda_1^{\pm 1}, (\lambda_0-\lambda_1)^{-1}]$
defined by
\begin{align}
g_2(x_1,x_2,x_3) = x_5^2 - x_4(x_4 - \lambda_0 d_1)(x_4 - \lambda_1 d_1) = 0.
\end{align}
Note that $U$ is a principal $\bCx$-bundle
over $\bP^1_{[\lambda_0:\lambda_1]} \setminus \{ 0, 1, \infty \}$.
The minimal model $Y_\lambda$
of the fiber $\varphi^{-1}(\lambda)$
over $\lambda = (\lambda_0,\lambda_1) \in U$
gives a $P$-polarized K3 surface.
The meromorphic differential form
\begin{align} \label{eq:Griffiths-Dwork}
\frac{\sum_{i=1}^5 (-1)^i q_i x_i d x_1 \wedge \cdots \wedge \widehat{d x_i} \wedge \cdots \wedge d x_5}
{g_2(x_1,x_2,x_3)(x_5^2 - x_4(x_4 - \lambda_0 d_1)(x_4 - \lambda_1 d_1))}
\end{align}
on $\bP(\bfq)$
gives a holomorphic volume form $\Omega_\lambda$
on $Y_\lambda$
as the iterated Poincar\'{e} residue,
and
the period map $\Pi$
is a map
from the universal cover $\Utilde$ of $U$
to $\cDt$
sending $\lambda$ to
the class of $\Omega_\lambda$
regarded as an element of $Q \otimes \bC \subset H^2(Y_\lambda, \bC)$.

Now
\pref{lm:weight}
follows from the $\bCx$-equivariance
\begin{align}
\Pi(\alpha^2 \lambda_0, \alpha^2 \lambda_1)
&= \Res \frac{\sum_{i=1}^5 (-1)^i q_i x_i d x_1 \wedge \cdots \wedge \widehat{d x_i} \wedge \cdots \wedge d x_5}
{g_2(x_1,x_2,x_3)(x_5^2 - x_4(x_4 - \alpha^2 \lambda_0 d_1)(x_4 - \alpha^2 \lambda_1 d_1))} \\
&= \Res \frac{\alpha^5 \sum_{i=1}^5 (-1)^i q_i x_i d x_1 \wedge \cdots \wedge \widehat{d x_i} \wedge \cdots \wedge d x_5}
{g_2(x_1,x_2,x_3)(\alpha^6 x_5^2 - \alpha^2 x_4(\alpha^2 x_4 - \alpha^2 \lambda_0 d_1)(\alpha^2 x_4 - \alpha^2 \lambda_1 d_1))} \\
&= \alpha^{-1} \Res \alpha \frac{\sum_{i=1}^5 (-1)^i q_i x_i d x_1 \wedge \cdots \wedge \widehat{d x_i} \wedge \cdots \wedge d x_5}
{g_2(x_1,x_2,x_3)(x_5^2 - x_4(x_4 - \lambda_0 d_1)(x_4 - \lambda_1 d_1))} \\
&= \alpha^{-1} \Pi(\lambda_0, \lambda_1)
\end{align}
of the period map,
where the passage from the first line to the second
comes from the change
$(x_1,x_2,x_3,x_4,x_5) \mapsto (x_1,x_2,x_3,\alpha^2 x_4, \alpha^3 x_5)$
of variables.
\end{proof}


It follows from \pref{eq:omega_bM},
\pref{eq:omega_bM2}, and
\pref{eq:omega_cM} that
\begin{align} \label{eq:omega}
\cO_{\bM}(-3) \cong (\phi^* \cO(-4))(\bH).
\end{align}
By taking the square of \pref{eq:omega}
and using \pref{eq:branch},
one obtains
\begin{align}
\cO_{\bM}(-6) \cong \cO_{\bM}(16-2 \deg \cH),
\end{align}
so that $\deg \cH=11$.
Since $\deg \cH_1=\cdots=\deg \cH_9 = 1$
and $\deg \cH_{10} = 2$,
one obtains
\begin{align}
\cH = \bigcup_{i=1}^{10} \cH_{i}.
\end{align}

The line bundle $\cO_\bM(-1)$ is characterized
as the cubic root of $\omega_\bM$;
\begin{align}
\cO_\bM(-1) \cong \phi^* \lb \cO_\cM(-5) \rb \lb \bH \rb.
\end{align}
One has
\begin{align}
\cO_\bM(-2) &\cong \phi^* \lb \cO_\cM(1) \rb, \\
\cO_\bM(-3) &\cong \phi^* \lb \cO_\cM(-4) \rb \lb \bH \rb, \\
&\vdots \\
\cO_\bM(-10) &\cong \phi^* \lb \cO_\cM(5) \rb, \\
\cO_\bM(-11) &\cong \cO_\bM (\bH),
\end{align}
so that
the graded ring
\begin{align}
A = \bigoplus_{k=0}^\infty H^0 \lb \cO_\bM(-k) \rb
\end{align}
of automorphic forms
is generated by four elements $\bfu_1, \bfu_2, \bfu_3, \bfu_4$ of weight 2
(corresponding to homogeneous coordinates $u_1, u_2, u_3, u_4$ of $\bP^3$)
and one element $\bfh$ of weight 11
(corresponding to the square root of
the defining equation $h \in \bC[u_1,u_2,u_3,u_4]$
of $\cH$)
with one relation $\bfh^2 = h(\bfu_1, \bfu_2, \bfu_3, \bfu_4)$
of weight 22.

\section{K3 surfaces of degree 2 with two rational double points of type $E_8$}
\label{sc:2+E8+E8}

Let $(Y, \cL)$ be a Gorenstein K3 surface of degree 2
with two rational double points of type $E_8$.
Let further $\mu \colon \Yt \to Y$ be the minimal model
and $\cL \coloneqq \mu^* \cL$.
If $\cLt \cong \cO_{\Yt}(s+2f)$
where $f$ is a fiber of an elliptic fibration
and $s$ is a section,
then the elliptic fibration must have two singular fibers of Kodaira type $\mathrm{I\!I}^*$.
If $\varphi_\cL \colon Y \to |\cL| \cong \bP^2$ is a double cover
branched over a sextic $C$,
then $C$ must have two singular points $p_1, p_2$ of type $E_8$,
and the composite $\pi \circ \varphi_\cL \circ \mu \colon \Yt \to \bP^1$,
where $\pi \colon |\cL| \dashrightarrow \bP^1$
is a projection from $p_1$,
gives an elliptic fibration with a section,
a singular fiber of Kodaira type $\mathrm{I\!I}^*$,
and another singular fiber of Kodaira type $\mathrm{I\!I\!I}^*$.
The first case is a degeneration of the second case,
and any Gorenstein K3 surface of degree 2
with two rational double points of type $E_8$
comes from an elliptic K3 surface
with a section,
a singular fiber of Kodaira type $\mathrm{I\!I}^*$
and another singular fiber
of Kodaira type equal to or worse than $\mathrm{I\!I\!I}^*$.
From the point of view of the Picard lattice,
this comes from an isometry
\begin{align} \label{eq:U+E8+E7}
P \coloneqq
\la 2 \ra \bot \bfE_8 \bot \bfE_8
\cong
\bfU \bot \bfE_8 \bot \bfE_7.
\end{align}
An embedding of the lattice $P$
into the K3 lattice $\LambdaK$
is unique up to isometry,
and the orthogonal lattice is
$Q \coloneqq \bfU \bot \bfU \bot \bfA_1$.
Since the automorphism group
of the discriminant group
$Q^\dual/Q \cong \bZ/2\bZ$
is trivial,
one has
$
\rOt(Q)
=
\rO(Q).
$

\begin{remark}
There exist isomorphisms
$
\rO^+(Q) \cong \Sp(4,\bZ)
$
and
$
\cD \cong \big\{ \tau \in \Mat_{2 \times 2}(\bC) \,\big|\,
\tau^{\mathrm T} = \tau
\text{ and } \Im \tau > 0 \big\}
$
which are compatible with the action.
\end{remark}

An elliptic K3 surface with a section
admits a Weierstrass model of the form
\begin{align} \label{eq:weierstrass1}
 z^2 = y^3 + g_8(x,w) y + g_{12}(x,w)
\end{align}
in $\bfP(1,4,6,1)$
(cf.~e.g.~\cite[Section 4]{MR2732092}).
Recall
(cf.~e.g.~\cite[Table IV.3.1]{MR1078016})
that the elliptic surface \eqref{eq:weierstrass1}
has a singular fiber of type
\begin{itemize}
\item
$\mathrm{I\!I\!I}^*$
at $0 \in \bP^1_{[x:w]}$
only if
$
\ord_0 g_8(x,w) = 3
$
and
$
\ord_0 g_{12}(x,w) \ge 5,
$
and
\item
$\mathrm{I\!I}^*$
at $\infty \in \bP^1_{[x:w]}$
only if
$
\ord_\infty g_8(x,w) \ge 4
$
and
$
\ord_\infty g_{12}(x,w) = 5,
$
\end{itemize}
so that
\begin{align}
 g_8(x,w) &= u_{4,4} x^4 w^4 + u_{3,5} x^3 w^5,
  \label{eq:weierstrass2} \\
 g_{12}(x,w) &= u_{7,5} x^7 w^5 + u_{6,6} x^6 w^6 + u_{5,7} x^5 w^7
  \label{eq:weierstrass3}
\end{align}
where $u_{7,5} \ne 0$.
The birational map
\begin{align}
\lb \frac{x}{w}, \frac{y}{w^4}, \frac{z}{w^6} \rb
\mapsto \lb \frac{x}{w^6}, \frac{y}{w^{14}}, \frac{z}{w^{21}} \rb
\end{align}
from $\bP(1,4,6,10)$
to $\bP(6,14,21,1)$
sends \pref{eq:weierstrass1}
to
\begin{align} \label{eq:weierstrass4}
z^2 = y^3 + g_{28}(x,w) y + g_{42}(x,w)
\end{align}
with
\begin{align}
g_{28}(x,w) &= u_{4,4} x^4 w^4 + u_{3,5} x^3 w^{10}, \\
g_{42}(x,w) &= u_{7,5} x^7 + u_{6,6} x^6 w^{6} + u_{5,7} x^5 w^{12}.
\end{align}
One can set $u_{7,5} = 1$
by rescaling $x$, $y$, and $z$,
and write
\begin{align}
g_{28}(x,w) &= t_4 x^4 w^4 + t_{10} x^3 w^{10}, \\
g_{42}(x,w) &= x^7 + t_6 x^6 w^{6} + t_{12} x^5 w^{12}.
\end{align}
This model has a singularity worse than
rational double points
on the fiber at $a \in \bfP^1$
if and only if $\ord_a(g_8) \ge 4$ and
$\ord_a(g_{12}) \ge 6$
(cf.~e.g.~\cite[Proposition I\!I\!I.3.2]{MR1078016}).
This can happen only for $a = 0$,
and this happens for $a=0$
if and only if
$t_{10}=t_{12}=0$.
The parameter
\begin{align}
 t = (t_4, t_6, t_{10}, t_{12})
  \in U \coloneqq \bC^4 \setminus
   \lc t_{10}=t_{12}=0 \rc
\end{align}
is unique up to the action of $\bCx$
given by
\begin{align}
 \bCx \ni \lambda \colon
  ((x,y,z,w),(t_i)_{i}) \mapsto ((x,y,z,\lambda^{-1} w), (\lambda^i u_i)_{i}),
\end{align}
which rescales the holomorphic volume form
\begin{align}
 \Omega= \Res \frac{w dx \wedge dy \wedge dz}{z^2 - y^3 - g_{28}(x,w;u) y - g_{42}(x,w;u)}
\end{align}
as
\begin{align} \label{eq:equivariance}
\Omega_{\lambda u}
= \lambda^{-1} \Omega_u.
\end{align}
The categorical quotient
$
 T \coloneqq U / \bCx_\mu
$
is the coarse moduli scheme of pairs
$(Y, \Omega)$ consisting of a $\bfU \bot \bfE_8 \bot \bfE_7$-polarized K3 surface $Y$
and a holomorphic volume form $\Omega \in H^0(\omega_Y)$ on $Y$.
The boundary of the affinization
$
 \Tbar \coloneqq \Spec \bC[T]
$
is given by
\begin{align}
 \lc t_{10} = t_{12} = 0 \rc \cong \bC_{t_4} \times \bC_{t_6}.
\end{align}
The coarse moduli space $\cM \cong \cD/\rO(Q)$
of pseudo-ample $P$-polarized K3 surfaces
is isomorphic to $T/\bCx$,
which is an open subvariety
of the weighted projective space
$\Tbar/\bCx \cong \bP(4,6,10,12)$.

The locus $\cH \subset \cM$
where the $P$-polarization is not ample
consists of two irreducible components;
$\cH = \cH_1 \cup \cH_2$.
The locus $\cH_1$ is where the singular fiber at $0 \in \bP^1_{[x:w]}$
is of Kodaira type $\mathrm{I\!I}^*$,
so that the Picard group contains $\bfU \bot \bfE_8 \bot \bfE_8$.
This divisor $\cH_1$
is defined by $t_{10} = 0$.
The other locus $\cH_2$ is where
there is a singular fiber of Kodaira type $\mathrm{I}_2$,
so that the Picard group contains $\bfU \bot \bfE_8 \bot \bfE_7 \bot \bfA_1$.
As explained in \cite[Section 6]{MR4356167},
this divisor is defined by
$k_{60}(t) \coloneqq \Delta_{120}(t)/r_{20}(t)^3$,
which is a weighted homogeneous polynomial of degree $60$,
where $\Delta_{120}$ and $r_{20}$ are defined as follows:
The discriminant
$4 g_{28}(x,w)^2 + 27 g_{42}(x,w)^3$
of the right hand side of \pref{eq:weierstrass4}
as a polynomial in $y$
is $x^9 w^{84}$ times a polynomial $f_{60}(x/w^6)$ of degree 5 in $x/w^6$,
which is homogeneous with respect to the weight of $t$
if we set the weight of $x/w^6$ to be $6$.
In other words, one has $f_{60}(x) = \prod_{i=1}^5 (x - \alpha_i)$
with $\deg \alpha_i = 6$.
It follows that the discriminant
$
\Delta_{120}(t) \coloneqq \prod_{1 \le i < j \le 5} (\alpha_i-\alpha_j)^2
$
of $f_{60}(x)$ is a weighted homogenous polynomial of degree
$2 \times \binom{5}{2} \times 6 = 120$ in $t$.
The resultant
\begin{align}
r_{20}(t) \coloneqq
\begin{vmatrix}
t_4 & t_{10} \\
 & t_4 & t_{10} \\
1 & t_6 & t_{12}
\end{vmatrix}
\end{align}
is a weighted homogeneous polynomial of degree $20$.
The locus defined by $\Delta_{120}(t)$
is where two singular fibers of Kodaira type $\mathrm{I}_1$ collide,
and the locus $r_{20}(t)$
is where there is a singular fiber of Kodaira type $\mathrm{II}$.

The orbifold quotient
$
\bM \coloneqq \ld \cD \middle/ \SO^+(Q) \rd
$
has a generic stabilizer of order 2
along the divisor $\bH = \bH_1 \cup \bH_2$
corresponding to $\cH = \cH_1 \cup \cH_2$.
One has
\begin{align}
\omega_\bM
&\cong \cO_\bM(-3), \\
\omega_\bM
&\cong \lb \phi^* \omega_\cM \rb (\bH), \\
\omega_\cM
&\cong \cO_\cM(-4-6-10-12)
\cong \cO_\cM(-32),
\end{align}
where $\phi \colon \bM \to \cM$ is the structure morphism
to the coarse moduli space
and
$\cO_\bM(\bH)$ is a line bundle
satisfying
\begin{align}
\cO_\bM(2 \bH)
&\cong \phi^* \lb \cO_\cM(\cH) \rb
\cong \phi^* \lb \cO_\cM(70) \rb.
\end{align}
It follows from \pref{eq:equivariance} that
\begin{align}
  \varphi^*(\cO_\cM(1)) \cong \cO_\bM(-1).
\end{align}
The line bundle $\cO_\bM(-1)$ is characterized
as the unique cubic root of $\omega_\bM$:
\begin{align}
\cO_\bM(-1)
\cong
\phi^*(\cO_\cM(-34))(\bH).
\end{align}
Note that
\begin{align}
H^0 \lb \cO_\bM(-2) \rb
\cong
H^0 \lb \phi^*(\cO_\cM(-68 + 70)) \rb
\cong
H^0 \lb \cO_\cM(2) \rb
\end{align}
and
\begin{align}
H^0 \lb \cO_\bM(-35) \rb
\cong
H^0 \lb \cO_\bM (\bH) \rb.
\end{align}
The graded ring $A(\rO^+(Q))$ of automorphic forms
is generated by
$\bft_4$,
$\bft_6$,
$\bft_{10}$,
$\bft_{12}$, and
$\bfs_{35}$
with relation
${\bfs_{35}}^2 = \bft_{10} k_{60}(\bft)$
\cite[Theorem 3]{MR168805}.

\section*{Acknowledgement}

I thank Kenji Hashimoto
for collaboration
at an early stage of this research;
this paper was originally conceived
as a joint project with him.
I was supported
by JSPS Grants-in-Aid for Scientific Research No.21K18575.

\bibliographystyle{amsalpha}
\bibliography{bibs}

\end{document}